\documentclass[amssymb,amsfonts]{amsart}
\usepackage{amssymb, latexsym}
\usepackage{amsrefs}
\usepackage[T1]{fontenc}
\usepackage[intlimits,sumlimits,noDcommand,narrowiints]{kpfonts}
\usepackage[english]{babel}
\usepackage{microtype}
\usepackage[pdfusetitle]{hyperref}

\theoremstyle{plain}
\newtheorem{thm}{Theorem}
\newtheorem{prop}[thm]{Proposition}
\newtheorem{lem}[thm]{Lemma}
\newtheorem{cor}[thm]{Corollary}
\theoremstyle{definition}
\newtheorem{defn}[thm]{Definition}

\theoremstyle{remark}
\newtheorem{rmk}[thm]{Remark}

\begingroup
\catcode`\%=12
\catcode`\$=12

\gdef\GitPageFooter{$Format: v:\texttt{
\endgroup

\begin{document}
\title{A commuting-vector-field approach to some dispersive estimates}
\author[WWY Wong]{Willie Wai Yeung Wong}
\address{Michigan State University}
\email{wongwwy@member.ams.org}

\begin{abstract}
We prove the pointwise decay of solutions to three linear equations: (i) the transport equation in phase space generalizing the classical Vlasov equation, (ii) the linear Schr\"odinger equation, (iii) the Airy (linear KdV) equation. 
The usual proofs use explicit representation formulae, and either obtain $L^1$---$L^\infty$ decay through directly estimating the fundamental solution in physical space, or by studying oscillatory integrals coming from the representation in Fourier space. 
Our proof instead combines ``vector field'' commutators that capture the inherent symmetries of the relevant equations with conservation laws for mass and energy to get space-time weighted energy estimates. 
Combined with a simple version of Sobolev's inequality this gives pointwise decay as desired.   
In the case of the Vlasov and Schr\"odinger equations we can recover sharp pointwise decay; in the Schr\"odinger case we also show how to obtain local energy decay as well as Strichartz-type estimates. 
For the Airy equation we obtain a local energy decay that is almost sharp from the scaling point of view, but nonetheless misses the classical estimates by a gap. 
This work is inspired by the work of Klainerman on $L^2$---$L^\infty$ decay of wave equations, as well as the recent work of Fajman, Joudioux, and Smulevici on decay of mass distributions for the relativistic Vlasov equation.
\end{abstract}

\maketitle

\section{Introduction}

This paper concerns pointwise decay estimates for dispersive partial differential equations. 
At the heart of the matter, we are interested in a classical field theory where the field strength measures the number density of the constituent ``particles''. 
That the equations of motion are ``dispersive'' indicates that individual ``particles'' tend to have disparate velocities, and as a result, will travel apart over time. 
As physically the total number of the particles are expected to be conserved, that the spatial support is spreading out in time suggests that the number density decreases in time. 
To realize this intuition, the classical proofs typically are based on analyses of the explicit representation formulae tying the field strengths at time $t$ to the field strengths at some initial time $t_0$.
The goal of this paper is to offer an alternative proof of some well-known dispersive inequalities using a method that bypasses the explicit representation formulae.  

This paper will focus on three examples, the classical (non-relativistic) Vlasov equation, the linear Schr\"odinger equation, and the Airy (linear Korteweg--de Vries) equation. The latter two will be introduced in Sections \ref{sec:sch} and \ref{sec:airy} respectively. We introduce the Vlasov equation here for illustration. The classical Vlasov equation is a simple linear transport equation on classical phase space. The field is the number density of a particle (say a gas) on the classical phase space $\mathbb{R}^d\times \mathbb{R}^d$. We use the coordinates $(q_1, \ldots, q_d, p_1, \ldots p_d)$; the first factor of $\mathbb{R}^d$ represents the position and the second factor the velocity. 
The (time-dependent) number density is given as 
\begin{equation}
\nu : \mathbb{R}\times\mathbb{R}^d\times\mathbb{R}^d \to [0,\infty)
\end{equation}
and we assume that the individual particles are non-interacting and hence follow Newton's first law
\begin{equation}\label{eq:vlasov}
\partial_t \nu + \underbrace{\sum_{i = 1}^d p_i \partial_{q_i} \nu}_{p\cdot \partial_q \nu} = 0.
\end{equation}
Equation \eqref{eq:vlasov} is sometimes called the classical Vlasov equation, and has an explicit solution of its initial value problem by the formula
\begin{equation}\label{eq:vlasovsol}
\nu(t,q,p) = \nu(0, q - tp, p).
\end{equation}
Using this formula, we can prove the following standard \emph{dispersive estimate}.
\begin{prop}\label{prop:vlasovdecay}
	Let $\overline{\nu}(t,q) := \int_{\mathbb{R}^d} \nu(t,q,p)~\mathrm{d}p$. If $\nu$ solves \eqref{eq:vlasov}, is smooth, and $\nu(0,q,p)$ decays suitably as $|p|,|q|\to\infty$, then 
	\begin{equation}
		\sup_{q\in\mathbb{R}^d} \overline{\nu}(t,q) \lesssim \langle t\rangle^{-d}.
	\end{equation}
	(The notation $\langle t\rangle := \sqrt{1 + t^2}$ will be in use throughout.)
\end{prop}
\begin{proof}
	By \eqref{eq:vlasovsol} we can write
	\[ \overline{\nu}(t,q) = \int_{\mathbb{R}^d} \nu(0,q - tp, p)~\mathrm{d}p.\]
	The integral on the right is over the $d$-dimensional hyperplane $\Pi(t,q) := \{ (q - tp,p) : p\in \mathbb{R}^d\}$ of $\mathbb{R}^d\times\mathbb{R}^d$. In terms of the induced hyperplane measure $\mathrm{d}\sigma$ on $\Pi(t,q)$, we see that the change of variables gives
	\[ \overline{\nu}(t,q) = \langle t\rangle^{-d} \int_{\Pi(t,q)} \nu(0, \text{---}) ~\mathrm{d}\sigma \]
	and hence the assertion is proved with the implicit constant 
	\[ \sup_{t,q} \int_{\Pi(t,q)} \nu(0,\text{---}) ~\mathrm{d}\sigma \]
	which can be bounded by $\|\nu(0,\text{---})\|_{W^{d,1}(\mathbb{R}^{2d})}$ by Gagliardo's Sobolev trace theorem. 
\end{proof}
The proof above captures many features of the representation-formula-based proofs of dispersive inequalities: the object to be controlled is written in terms of an explicit integral operator acting on the initial data, and the decay is read off from $L^1$---$L^\infty$ type bounds on the integral operator, with the asymptotics read off of the homogeneity properties of the integral operator (in other words, a change of variables). 
The analogous proofs (using the fundamental solution) for the Schr\"odinger and Airy equations can be found in Chapter 8 of Stein and Shakarchi \cite{PrinAnal4}.  The properties of the solution operators that are used in the course of proving the decay estimates can be derived from powerful oscillatory integral estimates from modern Fourier analysis. 

An alternative method for deriving dispersive decay estimates was found by Klainerman for the wave equation \cite{KlainermanSobolev}. Taking advantage of the Lorentz invariance of the wave equation, Klainerman observed that if the vector field $\Omega$ is a generator for the Poincar\'e group, and $u$ a solution to the wave equation, then $\Omega u$ is also a solution to the wave equation. From this the energy conservation of the wave equation implies certain space-time weighted energy inequalities for higher derivatives of $u$; and this, via a version of the Sobolev inequality, gives space-time weighted control of the $L^\infty$ norms of the solutions which is the pointwise decay estimates that we seek. 
More recently, Fajman, Joudioux, and Smulevici observed that by properly lifting the symmetry actions to the relativistic phase space, a similar argument can yield the dispersive decay for the \emph{relativistic} counterpart to the Vlasov equation \cite{VectorfieldVlasov}. 
The argument has been modified by Smulevici to apply to the classical Vlasov equation and was used to show small data global existence for the Vlasov-Poisson equations \cite{JacquesVlasov}.

Both the commuting vector field method and the traditional oscillatory integral approach for deriving dispersive estimates have many successes in their applications. Their relative merits have been explored in the literature (see, e.g.\ \cites{VectorfieldVlasov, KlainermanSobolev, KlainermanCommuting}) and we shall not discuss them here. 
In terms of the aim of providing a robust proof of dispersive inequalities that relies primarily on physical space methods (and avoids the use of the Fourier transform), there are also other previous works on bilinear estimates \cites{PlanchonVega, TaoUnp, PhysSpace}. 
The goal of the present article is to firstly demonstrate the feasibility of (re)deriving the analogues of certain classical dispersive estimates; secondly connect the commuting vector fields systematically to the symmetries of the equations; and thirdly relate the vector field commutators to the Fourier representation of the solutions, in the context of the three sample equations announced above. 
In our context, our equations exhibit symmetry properties that are Galilean or Galilean-like (in the sense that ``space'' and ``time'' are not on equal footing, as is in the case of Lorentzian symmetries). This makes decay estimates adapted to the standard $t$ foliation more obviously compatible with the vector field method; in the relativistic case one may argue that the estimates are more adapted to hyperboloidal foliations (see e.g.\ \cite{KlainermanKG}; and also \cites{QianPP, Hyperboloidal} for some recent developments). 
We fully exploit this compatibility for our relatively short proofs given below.

\section{Classical transport equations in phase space}

The classical Vlasov equation \eqref{eq:vlasov} is a special case of the more general class of transport equations on phase space.  
Let $\nu$ again denote the time-dependent number density on classical phase space $\mathbb{R}^d\times\mathbb{R}^d$. 
We let $w:\mathbb{R}^d\to\mathbb{R}^d$ be a smooth map and consider the following linear transport equation
\begin{equation}\label{eq:phasetransport}
	\partial_t \nu + w(p) \cdot \partial_q\nu = 0.
\end{equation}
The classical Vlasov equation \eqref{eq:vlasov} is simply \eqref{eq:phasetransport} with $w$ being the identity function. 

Before treating \eqref{eq:phasetransport} more generally, let us focus first on the case of the classical Vlasov equation. 
This case has been previously treated by Smulevici \cite{JacquesVlasov}, we include the discussion here to set the stage for the general case, and to showcase how the analysis simplifies due to the Galilean (instead of Lorentzian) symmetry of the problem. 
The $t$-weights in the weighted energy estimates that drive both the temporal decay for the linear wave equation in the original Klainerman-Sobolev estimate \cite{KlainermanSobolev} and the analogue for the relativistic Vlasov equation are derived from the Lorentz-boost vector fields. 
Here, for the classical Vlasov equation, we will instead take advantage of the \emph{Galilean boosts}: 
if $\nu$ solves \eqref{eq:vlasov}, then so does the function 
\[ (t,q,p) \mapsto \nu(t, q + t p_0, p + p_0)\]
for any $p_0\in \mathbb{R}^d$. The corresponding infinitesimal generators of these symmetries are given by the vector fields $W_i := t \partial_{q_i} + \partial_{p_i}$, where $i\in \{1, \ldots, d\}$. That is to say, if $\nu$ is a solution to \eqref{eq:vlasov} then so is $W_i \nu$. We see that $W_i$ has an obvious $t$-weight; this is the factor that will drive the decay for large times. 
The dispersive estimate of Proposition \ref{prop:vlasovdecay} then follows from the following two Lemmas. 

\begin{lem}[Conservation laws] \label{lem:ptconslaw}
	If $\nu$ solves \eqref{eq:phasetransport}, and $F:\mathbb{R}^d \times \mathbb{R} \to \mathbb{R}$, then 
	\[ \iint_{\mathbb{R}^d\times\mathbb{R}^d} F(p, \nu(t,q,p)) ~\mathrm{d}p~\mathrm{d}q \]
	is constant in time when it is well-defined. 
\end{lem}
\begin{proof}[Sketch of proof]
	When $F$ is differentiable in the second factor, then $F(p,\nu)$ is a classical solution also to \eqref{eq:phasetransport} which is a conservation law in divergence form. Provided $F(p,\nu)$ decays suitably at infinity the spatial integral $\int_{\mathbb{R}^d} w(p)\cdot \partial_q F(p,\nu) ~\mathrm{d}q$ vanishes by the divergence theorem, and the conservation law holds. For more general $F$ we approximate by mollified versions. 
\end{proof}

\begin{lem}[``Klainerman-Sobolev'' for classical Vlasov]
	If $\nu$ solves \eqref{eq:vlasov}, then 
	\[ |t|^d \|\overline{\nu}(t,\text{---})\|_{L^\infty(\mathbb{R}^d)} \leq \iint_{\mathbb{R}^d\times\mathbb{R}^d} \left| W_1 W_2 \cdots W_d \nu(t,q,p)\right| ~\mathrm{d}p~\mathrm{d}q.\]
\end{lem}
\begin{proof}
	Writing $Q(q) := (-\infty,q_1)\times(-\infty,q_2)\times\cdots\times (-\infty,q_d)$ for the orthant below $q$, the fundamental theorem of calculus, applied to $\nu$ which we assume to decay suitably at infinity, yields 
	\[ \overline{\nu}(t,q) = \int_{Q(q)} \partial_{q_1}\partial_{q_2}\cdots\partial_{q_d}\overline{\nu}(t,q') ~\mathrm{d}q'.\]
	Next, observing that if $\nu$ decays suitably at infinity, 
	\[ \int_{\mathbb{R}^d} \partial_{p_i} \nu(t,q,p')~\mathrm{d}p' = 0.\]
	This implies that
	\[ t^d \overline{\nu}(t,q) = \int_{Q(q)} \int_{\mathbb{R}^d} W_1 \cdots W_d \nu(t,q',p') ~\mathrm{d}p'~\mathrm{d}q'\]
	and the lemma follows. 
\end{proof}

Putting together the two lemmas, we have that 
\begin{multline*} 
	|t|^d \|\overline{\nu}(t,\text{---})\|_{L^\infty(\mathbb{R}^d)} \leq \iint_{\mathbb{R}^d\times\mathbb{R}^d} \left| W_1\cdots W_d \nu(0,q,p)\right| ~\mathrm{d}p~\mathrm{d}q \\
	= \iint_{\mathbb{R}^d\times\mathbb{R}^d} |\partial_{p_1}\cdots\partial_{p_d} \nu(0,q,p)|~\mathrm{d}p~\mathrm{d}q \leq \|\nu(0,\text{---})\|_{W^{d,1}(\mathbb{R}^{2d})}.
\end{multline*}
This gives Proposition \ref{prop:vlasovdecay} for $|t| \geq 1$. For $|t| \leq 1$ we use that spatial translations $\partial_{q_i}$ are also Galilean symmetries, and hence 
\[ \|\overline{\nu}(t,\text{---})\|_{L^\infty(\mathbb{R}^d)} \leq \iint_{\mathbb{R}^d \times\mathbb{R}^d} \left| \partial_{q_1}\cdots\partial_{q_d} \nu(t,q,p)\right| ~\mathrm{d}p~\mathrm{d}q \]
with the right hand side being a conserved quantity controlled also by $\|\nu(0,\text{---})\|_{W^{d,1}(\mathbb{R}^{2d})}$.

\begin{rmk}
	This argument illustrates the basic structure of a proof using the commuting vector field method. Using the symmetries of the equation one gets space-time weighted integral conservation laws. This conservation law is converted into a point-wise decay estimate by way of a Sobolev inequality.
\end{rmk}

Returning to the more general case \eqref{eq:phasetransport}, we see immediately that if the Jacobian matrix of the mapping $w$ is nonsingular, then we can invert the mapping and consider $\nu$ as a function of $(t,q,w)$. In this case \eqref{eq:phasetransport} reduces to the classical Vlasov equation and the above arguments go through \emph{mutatis mutandis} giving $\langle t\rangle^{-d}$ decay for the solutions. 
If we let $w(p) = p / \sqrt{1 + |p|^2}$ for example, the equation \eqref{eq:phasetransport} becomes the relativistic Vlasov equation on Minkowski space. So this gives an alternative proof for its dispersive decay. (Note that the velocity integral for $\overline{\nu}$ would also have to be suitably modified; see the much more exhaustive treatment in \cite{VectorfieldVlasov}.)

The situation becomes somewhat more interesting when $w$ has critical points. We start with an example. 
\begin{prop}
	Let $d = 1$ and $w = p^2$, then for every $C, \epsilon > 0$, there exists a smooth solution $\nu$ of \eqref{eq:phasetransport} and a large time $T$ such that
	 \[ \overline{\nu}(T,0) \geq C \langle T\rangle^{-\epsilon} \|\nu(0,\text{---})\|_{W^{1,1}(\mathbb{R}^2)}.\]
\end{prop}
\begin{proof}
	Fix $\phi$ to be a radial bump function on $\mathbb{R}^2$ that is identically $1$ in the unit disc and vanishes outside the disc of radius 2. Denote by $\phi_\lambda(q,p) = \lambda \phi(\lambda q,\lambda p)$; for all $\lambda \geq 1$ we have, by scaling, that $\|\phi_\lambda\|_{W^{1,1}(\mathbb{R}^2)}$ is uniformly bounded by some constant $C'$. Let $\nu(t,q,p) = \phi_\lambda(q - tp^2, p)$; this solves \eqref{eq:phasetransport}. We have the following lower bound
	\[ \overline{\nu}(t,0) \geq \lambda \int_{\{p^2 + t^2 p^4 < \lambda^{-2}\}} ~\mathrm{d}p = \lambda \sqrt{ \frac{\sqrt{ 4 \lambda^{-2} t^2 + 1} - 1}{2 t^2}}.\]
	This implies, in particular, that 
	\[ \overline{\nu}(\lambda,0) \geq \sqrt{ \frac{\sqrt{5} - 1}{2}}.\]
	So choosing $T = \lambda$ sufficiently large such that for our given $C$ and $\epsilon$ the inequality
	\[ C C' \langle T\rangle^{-\epsilon} \leq \frac12\]
	holds we obtain the desired counterexample. 
\end{proof}

\begin{rmk}
	The equation 
	\[ \partial_t \nu + p^2 \partial_q \nu = 0 \]
	considered in the previous proposition is the classical-phase-space-transport analogue of the Airy equation 
	\[ \partial_t \phi - \partial^3_{xxx} \phi = 0.\]
	This latter equation has a well-known $L^1$---$L^\infty$ decay estimate with a $\langle t\rangle^{-1/3}$ rate \cite{PrinAnal4}*{Chapter 8}. Here we see a difference between the \emph{classical} and the \emph{quantum} pictures: in the latter a critical point of the dispersive relation $w$ gives a reduced rate of decay, in the former such critical points invalidate the decay estimates entirely. This difference can be understood in part by the Heisenberg uncertainty principle for quantum systems, which disallows initial data like $\phi_\lambda$ which concentrates both in physical and frequency space. 
\end{rmk}

Let us now return to the general equation \eqref{eq:phasetransport} in arbitrary spatial dimensions $d$. First, we notice that spatial translation remain a symmetry of these equations. Therefore in the spirit of vector field method we observe that for any multiindex $\alpha$, we have that, by Lemma \ref{lem:ptconslaw}, the integral
\[ \int_{\mathbb{R}^d\times\mathbb{R}^d} \left| \partial^\alpha_q \nu(t,q,p) \right| ~\mathrm{d}p~\mathrm{d}q \]
is conserved in time. Therefore by the Sobolev inequality we have \emph{boundedness} of solutions to \eqref{eq:phasetransport}:
\begin{equation}
	\overline{\nu}(t,q) \leq \|\nu(0,\text{---})\|_{W^{d,1}(\mathbb{R}^{2d})}.
\end{equation}
This can be upgraded to partial decay provided the singularity in the dispersion relation is not too bad. An example is the following.
\begin{defn}
	We say that the mapping $w:\mathbb{R}^d\to\mathbb{R}^d$ \emph{has rank at least $k$} if there exists a locally finite cover of $\mathbb{R}^d$ by open sets $U_\alpha$ such that for every $\alpha$, there exists a matrix valued function $B_\alpha:U_\alpha \to \mathbb{M}^{d\times d}$ and a $k$-dimensional subspace $V_\alpha$ of $\mathbb{R}^d$ such that at every point $p\in U_\alpha$, the matrix product $B_\alpha\cdot \partial w$ is the projection from $\mathbb{R}^d\to V_\alpha$. 
\end{defn}
\begin{thm}\label{thm:degenerate}
	If the mapping $w$ has rank at least $k$, then solutions of \eqref{eq:phasetransport} verify the decay estimate
	\[ \overline{\nu}(t,q) \leq \langle t\rangle^{-k} \|\nu(0,\text{---})\|_{W^{d,1}(\mathbb{R}^{2d}, \varpi(p) ~\mathrm{d}p~\mathrm{d}q)},\]
	where the right hand side denotes the weighted Sobolev space with some weight $\varpi(p)$ which depends on $w$ but not on the solution $\nu$. 
\end{thm}
\begin{proof}
	First observe that the analogues of the Galilean symmetry vector fields are $W_i := \partial_{p_i} + t \sum_{j = 1}^d (\partial_{p_i} w^j )\partial_{q_j}$. If $\nu$ solves \eqref{eq:phasetransport} then so does $W_i \nu$. By linearity, we see also that for any set of functions $f_i:\mathbb{R}^d \to\mathbb{R}$ and $g:\mathbb{R}^d\to\mathbb{R}$, the function 
	\[ \sum_{i = 1}^{d} f_i(p) W_i\nu(t,q,p) + g(p) \nu(t,q,p) \]
	also solves \eqref{eq:phasetransport}. 

	By assumption there exists a preferred locally finite cover $U_\alpha$ of $\mathbb{R}^d$. Let $\chi_\alpha$ denote a subordinate partition of unity, and let $B_\alpha$ be the corresponding matrix valued functions. Now fix $\alpha$. Without loss of generality we can assume that $V_\alpha$ is equal to the span of $\{e_1, \ldots, e_k\}$ the first $k$ standard vectors. Then for $\ell \in \{1, \ldots, k\}$ we can define 
	\[ \tilde{W}_\ell \nu_{\alpha} := \sum_{i = 1}^d B_{\alpha,\ell i} W_i \nu_{\alpha} + (\partial_{p_i} B_{\alpha,\ell i})\nu_{\alpha} ,\]
	where
	\[ \nu_\alpha(t,q,p) := \chi_\alpha(p) \nu(t,q,p).\]
	Now, since $\nu_\alpha$ is obtained from $\nu$ with a velocity cut-off, whenever $\nu$ solves \eqref{eq:phasetransport} so does $\nu_\alpha$. Then the discussion at the beginning of this proof shows that $\tilde{W}_\ell \nu_{\alpha}$ is a solution also. Using the properties of $B_\alpha \cdot \partial w$ as a projection, we have that in fact
	\[ \tilde{W}_\ell \nu_{\alpha} 	= t \partial_{q_{\ell}} \nu_{\alpha} + \sum_{i = 1}^d \partial_{p_i} \left( B_{\alpha,\ell i} \nu_\alpha\right).\]
	
	This allows us to write 
	\[ t^k \overline{\nu_\alpha}(t,q) = \int_{Q(q)} \int_{U_\alpha} \tilde{W}_1 \tilde{W}_2\cdots \tilde{W}_k \partial_{q_{k+1}} \cdots \partial_{q_d} \nu_\alpha(t,q',p) ~\mathrm{d}p ~\mathrm{d}q'\]
	where $Q(q) = (-\infty,q_1)\times\cdots \times(-\infty,q_d)$ as before. This implies
	\[ t^k \overline{\nu_\alpha}(t,q) \leq \iint_{\mathbb{R}^d \times U_\alpha}  \left| \tilde{W}_1 \tilde{W}_2\cdots \tilde{W}_k \partial_{q_{k+1}} \cdots \partial_{q_d} \nu_\alpha(t,q',p)\right| ~\mathrm{d}p ~\mathrm{d}q'.\]
	The integral on the right is a conserved quantity, and so is entirely determined by the initial data. In particular, this shows that there exists some function $\varpi_\alpha:U_\alpha\to\mathbb{R}_+$ such that 
	\[ t^k \overline{\nu_\alpha}(t,q) \leq \|\nu(0,\text{---})\|_{W^{d,1}(\mathbb{R}^d\times U_\alpha, \varpi_{\alpha}(p)~\mathrm{d}p~\mathrm{d}q)}.\]
	Now, noting that 
	\[ \overline{\nu}(t,q) = \sum_{\alpha} \overline{\nu_\alpha}(t,q) \]
	by our partition of unity, we have that there exists some $\varpi:\mathbb{R}^d\to \mathbb{R}_+$ such that 
	\[ t^k \overline{\nu}(t,q) \leq \|\nu(0,\text{---})\|_{W^{d,1}(\mathbb{R}^{2d}, \varpi(p) ~\mathrm{d}p~\mathrm{d}q)}.\]
	Interpolating with the boundeness we have the result as claimed. 
\end{proof}

\begin{rmk}
	Theorem \ref{thm:degenerate} should be compared with Fourier restriction theorems to submanifolds with some degree of degeneracy in the curvature. A classical example of this scenario is that corresponding to the decay estimates for the linear wave equation \cite{SchlagMusculu}*{Section 11.3.4}.
\end{rmk}

\begin{rmk}
	The velocity cut-off used in the proof above is analogous to \emph{frequency cut-offs} in the study of solutions to constant coefficient partial differential equations. 
\end{rmk}

\section{Linear Schr\"odinger equation}\label{sec:sch}

Having treated the classical Vlasov equation and its cousins, let us now move our attention to the linear Schr\"odinger equation
\begin{equation}\label{eq:ls}
	\partial_t u + i\triangle u = 0
\end{equation}
where $u:\mathbb{R}\times\mathbb{R}^d\to\mathbb{C}$. Equation \eqref{eq:ls} can be understood as the quantum analogue of \eqref{eq:vlasov}: indeed we can simple-mindedly obtain Schr\"odinger's equation from the classical Vlasov equation by using the quantization $p \mapsto  i \partial_q$ relating the classical and quantum phase spaces. 

Our lesson from Vlasov equation suggests that we should look to using the Galilean boost in our vector field method. Our quantization procedure suggests that the correct linear operator should be
\begin{equation}
	W_j: u \mapsto t \partial_{q_j}u + \frac{i}{2} q_ju.
\end{equation}
Indeed one can check that if $u$ solves \eqref{eq:ls} then so does $W_j u$. The associated Klainerman-Sobolev estimate is
\begin{lem}[``Klainerman-Sobolev'' for Schr\"odinger] \label{lem:ksforls}
	Let $u$ be a smooth solution of \eqref{eq:ls} such that the trace $u(t,\text{---})$ for every $t$ is in Schwartz space. Then there exists a constant $C$ depending only on the dimension $d$ such that 
	\[ |t|^{d} \|u(t,\text{---})\|_{L^\infty(\mathbb{R}^d)}^2 \leq C \sum_{|\alpha| + |\beta| = d} \|W^\alpha u(t,\text{---})\|_{L^2(\mathbb{R}^d)}\|W^\beta u(t,\text{---})\|_{L^2(\mathbb{R}^d)}\]
	where $\alpha,\beta$ are multiindices and if $\alpha = (\alpha_1, \alpha_2, \cdots \alpha_d)$ we have the operator $W^{\alpha} = W_1^{\alpha_1} W_2^{\alpha_2} \cdots W_d^{\alpha_d}$.
\end{lem}
\begin{rmk}
	Note that $[W_j,W_k] = 0$, so the order in which the components of $W^{\alpha}$ are listed does not matter. 
\end{rmk}
\begin{proof}
	Letting again $Q(q)$ be the orthant below $q$, we note that (here $\bar{u}$ denotes the complex conjugate)
	\[ u(t,q)\bar{u}(t,q) = \int_{Q(q)} \partial_1\partial_2\cdots \partial_d [u(t,q')\bar{u}(t,q')]~\mathrm{d}q'. \]
	Next notice that we have the Leibniz-like rule
	\[ t \partial_j [u \bar{v}] = t\bar{v} \partial_j u + t u \partial_j\bar{v} = \bar{v} W_j u + u \overline{W_j v}.\]
	So our lemma follows from Cauchy-Schwarz. 
\end{proof}

To better capture the decay properties, we introduce the dyadic norm $X^{\theta,q}$: let $\phi_k$ denote a sequence of bump functions such that 
	\begin{itemize}
		\item $\sum_{k\in\mathbb{Z}} \phi_k \equiv 1$;
		\item $\phi_k$ is supported in the annulus of of inner radius $2^{k-1}$ and outer radius $2^{k+1}$;
		\item $\phi_k$ is smooth, real-valued, and non-negative.
	\end{itemize}
We define
\begin{equation}\label{eq:defxtheta}
	\| f\|_{X^{\theta,q}} := \left\| \left(2^{\theta k} \| \phi_k f\|_{L^2}\right)_{k\in\mathbb{Z}} \right\|_{\ell^q}.
\end{equation}
Quite obviously, we have
\begin{equation}
	\|f\|_{X^{\theta,2}} \approx \| ~|\text{---}|^\theta f(\text{---}) \|_{L^2}.
\end{equation}

\begin{thm}[Dispersive estimate for Schr\"odinger] \label{thm:displs}
	There is a constant $C$ such that every solution $u$ 
	of \eqref{eq:ls} such that the trace $u(t,\text{---})$ for every $t$ is in Schwartz space satisfies 
	\[ |t|^{d/2} \|u(t,\text{---})\|_{L^\infty(\mathbb{R}^d)} \leq C \left\| u(0,\text{---})\right\|_{X^{d/2,1}(\mathbb{R}^d)}.\]
\end{thm}

\begin{proof}
	Using that $W_ju$ solves also \eqref{eq:ls} and that the $L^2$ mass is conserved for solutions of Schr\"odinger equation, we see that Lemma \ref{lem:ksforls} implies
	\begin{equation} \label{eq:firstdis}
		|t|^{d} \|u(t,\text{---})\|_{L^\infty}^2 \leq C \sum_{|\alpha| + |\beta| = d} \| (\text{---})^\alpha u(0,\text{---})\|_{L^2}\|(\text{---})^\beta u(0,\text{---})\|_{L^2}. 
	\end{equation}
	This implies directly that
	\[ |t|^{d/2} \|u(t,\text{---})\|_{L^\infty} \leq C \| \langle \text{---}\rangle^d u(0,\text{---})\|_{L^2}\]
	which, while does in fact give the correct temporal decay, has a spatial weight that is too strong compared to scaling (see Remark \ref{rmk:scaling} below). To tighten the weights we use our dyadic decomposition. 
	We write for $u_k$ the solution to \eqref{eq:ls} with initial data $u_k(0,q) = \phi_k(q) u(0,q)$. By linearity we have that 
	\[ u = \sum_{k\in\mathbb{Z}} u_k.\]
	Equation \eqref{eq:firstdis} implies
	\begin{align*}
		|t|^{d/2} \|u(t,\text{---})\|_{L^\infty} & \leq t^{d/2} \sum_{k\in\mathbb{Z}} \|u_k(t,\text{---})\|_{L^\infty}\\
		& \lesssim_{d} \sum_{k\in\mathbb{Z}} \left( \sum_{|\alpha| + |\beta| = d}   \| (\text{---})^\alpha u_k(0,\text{---})\|_{L^2}\|(\text{---})^\beta u_k(0,\text{---})\|_{L^2} \right)^{\frac12}
	\end{align*}
	Using the restricted spatial support, we have that
	\[ \| (\text{---})^\alpha u_k(0,\text{---})\|_{L^2} \leq 2^{(k+1)|\alpha|} \|u_k(0,\text{---})\|_{L^2}\]
	so
	\[ |t|^{d/2} \|u(t,\text{---})\|_{L^\infty} \lesssim_d \sum_{k\in \mathbb{Z}} 2^{kd/2} \|u_k(0,\text{---})\|_{L^2}\]
	as claimed. 
\end{proof}
\begin{rmk}\label{rmk:scaling}
	The classical dispersive inequality for Schr\"odinger's equation takes the form \cite{PrinAnal4}
	\[ |t|^{d/2} \|u(t,\text{---})\|_{L^\infty} \lesssim_d \|u(0,\text{---})\|_{L^1}.\]
	One easily checks that $X^{d/2,1}$ embeds strictly into $L^1$, so Theorem \ref{thm:displs} follows from the classical dispersive inequality for Schr\"odinger's equation. We shall show later that Theorem \ref{thm:displs} also \emph{implies} the classical $L^1$--$L^\infty$ estimate, in spite of the fact that there exists $L^1$ functions not in $X^{d/2,1}$. 
\end{rmk}

\begin{rmk}
	The dyadic decomposition in physical space that is used to recover the correct scaling of $|q|^{d/2}$ from the more lossy naive estimate \eqref{eq:firstdis} is reminiscent of an argument given by Klainerman in \cite{KlainermanCommuting}. There, the aim is to recover the Strichartz estimate for wave equations from the Klainerman-Sobolev inequalities. The linear dispersive estimate for wave equations, however, has some built-in smoothing property that appears to gain $(d-1)/2$ derivatives (on the $L^1$ scale) compared to (Klainerman-)Sobolev. Klainerman overcame this by a phase-space localization procedure. 

	(We note in passing that the estimate in Theorem \ref{thm:displs} also exhibits smoothing. Compared to Sobolev embedding it gains $d/2$ derivatives in the $L^2$ scale.)

	In terms of the pointwise estimate, one should note that the dyadic decomposition of initial data in physical space is related to a dyadic decomposition in frequency space. The intuition from the Vlasov equation suggests that the ``wave packets'' which contribute to the field at time $t$ at $q = 0$ that originated from time $0$ at $|q| \approx 2^k$ will have velocity $\approx 2^{k} / t$. That is to say, we expect $u(1,0) \approx \sum P_k u_k(1,0)$ where $P_k$ is the standard Littlewood-Paley projector. So frequency space and physical space decompositions are expected to have similar effects in the course of this proof. 
\end{rmk}

With real interpolation (see \cite{Interpolation}*{Chapter 5} for the results needed) we have, as an immediate corollary, the following result.
\begin{cor}\label{cor:untruncated}
	For every $\theta\in [0,1]$ there exists $C$ depending on $\theta$ such that the estimate
	\[ |t|^{\theta d / 2} \|u(t,\text{---})\|_{L^{2/(1-\theta)}} \leq C \|u(0,\text{---})\|_{X^{\theta d/2,2}} \]
	holds for solutions $u$ of \eqref{eq:ls}.
\end{cor}

To control the $L^p$ norms on time slices when $|t|$ is small, we can use the Sobolev embedding $H^{\theta d/2}(\mathbb{R}^d) \hookrightarrow L^{2/(1-\theta)}(\mathbb{R}^d)$ for $\theta\in [0,1)$. The conservation of the $H^s$ norms for the Schr\"odinger equation implies then
\begin{cor}\label{cor:truncated}
	For every $\theta\in[0,1)$, there exists $C$ depending on $\theta$ such that the estimate
	\[ \langle t\rangle^{\theta d/2} \|u(t,\text{---})\|_{L^{2/(1-\theta)}} \leq C \left[ 
	    \|\; |\text{---}|^{\theta d/2} u(0,\text{---}) \|_{L^2} + \| u(0,\text{---})\|_{H^{\theta d/2}} \right] \]
	holds for solutions $u$ of \eqref{eq:ls}.
\end{cor}
\begin{rmk}
	Note that Corollary \ref{cor:truncated} does not apply to the end-point $L^\infty$ case due to the failure of the Sobolev embedding in that case, as well as the failure of having an $X^{d/2,2}$ estimate (we only have $X^{d/2,1}$).  Corollaries \ref{cor:untruncated} and \ref{cor:truncated} should be compared with Theorem \ref{thm:degenerate}.
	In Theorem \ref{thm:degenerate} ``regularity in the $p$ direction (velocity/frequency space)'' is what guarantees long-time decay, while ``regularity in the $q$ direction (position/physical space)'' is what guarantees short-time boundedness. Analogously, in Corollary \ref{cor:untruncated} it is weights in physical space (which by the Fourier transform is equivalent to regularity in frequency space in the quantum picture) that guarantees the long-time decay of solutions, while regularity in physical space is again used to guarantee short-time boundedness in Corollary \ref{cor:truncated}. That Theorem \ref{thm:degenerate} can get the end-point $L^\infty$ estimate is down to our working in $L^1$ instead of $L^2$ based spaces in that scenario. 
\end{rmk}
\begin{rmk}
	The conservation of $H^s$ norms for solutions to \eqref{eq:ls} is part of the more general fact that if $T$ is a Fourier multiplier then $\|T u(t,\text{---})\|_{L^2}$ is conserved for \eqref{eq:ls}. This fact is obvious using the Fourier representation of the solutions, and should be compared to Lemma \ref{lem:ptconslaw} for the Vlasov equation. 
\end{rmk}


We next prove Strichartz-type estimates. Noting
\[ \| \phi_k f\|_{L^2} \leq C \cdot 2^{kd/2} \| f\|_{L^\infty} .\]
we have that
\begin{equation}
	\|f \|_{X^{-d/2,\infty}} \lesssim \|f\|_{L^\infty}.
\end{equation}
So by Theorem \ref{thm:displs} we get that 
\begin{equation}
	|t|^{d/2} \| U(t)f \|_{X^{-d/2,\infty}} \leq C \| f\|_{X^{d/2,1}}.
\end{equation}
On the other hand, mass conservation gives
\begin{equation}
	\|U(t) f\|_{X^{0,2}} = \|f\|_{X^{0,2}}.
\end{equation}
Interpolating between the two (see \cite{Interpolation}*{Theorem 5.6.1}) we obtain
the following lemma.
\begin{lem}\label{lem:localmass}
	For every $\sigma\in [0,d/2)$, there exists a constant $C$ such that for every $f\in \mathcal{S}(\mathbb{R}^d)$
	\[ |t|^\sigma \|U(t)f\|_{X^{-\sigma,2}} \leq C \|f\|_{X^{\sigma,2}}.\]
\end{lem}
\begin{rmk}
	This lemma, and the Strichartz-type estimate to be given below, are really statements concerning time-decay and integrability (as a function of time) of local mass. 
	Letting $\sigma = \theta d/2$ for $\theta\in [0,1)$, we see that the norm $X^{-\sigma,2}$ has the same scaling as $L^{2/(1-\theta)}$ that appears in Corollary \ref{cor:untruncated}; the two norms are, however, not comparable. In terms of scaling, this lemma and the Strichartz estimate to follow are sharp. 
\end{rmk}

Now, letting $\Phi, \Psi$ be functions on $\mathbb{R}\times\mathbb{R}^d$, Lemma \ref{lem:localmass} implies that for $\sigma\in [0,d/2)$:
\begin{equation}\label{eq:bilinearpttime}
	 |t-s|^{\sigma} \left\langle U(t)^* \Phi(t,\text{---}), U(s)^* \Psi(s,\text{---}) \right\rangle \lesssim_{d,\sigma} \|\Phi(t,\text{---})\|_{X^{\sigma,2}} \|\Psi(s,\text{---})\|_{X^{\sigma,2}},
\end{equation}
where $\langle \text{---},\text{---}\rangle$ denotes the $L^2(\mathbb{R}^d,\mathbb{C})$ pairing. Recall now the Hardy-Littlewood-Sobolev lemma, which states that when $g = |\text{---}|^{-\sigma} * f$ are functions on the real line, then
\[ \|g\|_{L^q(\mathbb{R})} \lesssim \|f\|_{L^p(\mathbb{R})} \]
when $q > p > 1$ and $0 < \sigma = 1 + q^{-1} - p^{-1}$. Applying to the case $q^{-1} + p^{-1} = 1$ which requires $p = 2/(2-\sigma)$, we get from \eqref{eq:bilinearpttime}
\begin{equation}
	\iint_{\mathbb{R}^2} \left\langle U(t)^* \Phi(t,\text{---}), U(s)^* \Psi(s,\text{---}) \right\rangle ~\mathrm{d}s~\mathrm{d}t \lesssim_{d,\sigma} \|\Phi(t,\text{---})\|_{L^p_tX^{\sigma,2}} \|\Psi(s,\text{---})\|_{L^p_sX^{\sigma,2}}.
\end{equation}
So by the $TT^*$ argument we get finally
\begin{thm}[$X^{\theta,q}$ Strichartz-type inequalities for Schr\"odinger]
	The Schr\"odinger propagator $U(t)$ satisfies
	\begin{align*}
	\| U(t)\phi\|_{L^{p'}_t X^{-\sigma,2}} &  \lesssim_{d,\sigma} \|\phi\|_{L^2(\mathbb{R}^d)}, \\
	\left\| \int_{\mathbb{R}} U(s)^*\Phi(s,\text{---}) ~\mathrm{d}s \right\|_{L^2(\mathbb{R}^d)} & \lesssim_{d,\sigma} \|\Phi\|_{L^p_t X^{\sigma,2}},
	\end{align*}
	provided $(p,p',\sigma)$ satisfies 
	\[ \frac{1}{p} + \frac{1}{p'} = 1, \qquad 1 < p = \frac{2}{2-\sigma} < 2.\]
\end{thm}

\begin{rmk}
	We can also recover the standard $L^p$ decay estimates from Theorem \ref{thm:displs}, which leads also to a proof of the standard Strichartz inequality. We claim that optimizing Theorem \ref{thm:displs} allows us to show that in fact  
	\begin{equation}\label{eq:stdlpest}
		|t|^{d/2} \|u(t, \text{---})\|_{L^\infty(\mathbb{R}^d)} \leq C \| u(0, \text{---}) \|_{L^1(\mathbb{R}^d)}.
	\end{equation}
	The main idea is to exploit the fact that the $L^\infty$ norm is translation invariant, but not the norm $\left\| |\text{---}|^{d/2} u(0,\text{---}) \right\|_{L^2}$. Denoting by $\tau_y$ the translation operator
	\[ \tau_y f(x) = f(x + y), \]
	we can optimize Theorem \ref{thm:displs} to read 
	\[ |t|^{d/2} \|u(t,\text{---})\|_{L^\infty(\mathbb{R}^d)} \leq C \inf_{y\in \mathbb{R}^d} \left\| \tau_y u(0,\text{---})\right\|_{X^{d/2,1}}.\]
	Applying this estimate to $u(0,x)$ being the characteristic function of any cube in $\mathbb{R}^d$, we note that the infimum is bounded above by the case when the translation brings the center of the cube to the origin, in which case a direct computation yield that 
	\begin{equation}\label{eq:estforcubes}
		 \inf_{y\in \mathbb{R}^d} \left\| \tau_y u(0,\text{---})\right\|_{X^{d/2,1}} \leq C \|u(0, \text{---})\|_{L^1(\mathbb{R}^d)}.
	\end{equation}
	Finally, by linearity of the equation we can approximate arbitrary initial data by simple functions, and use the uniform bound \eqref{eq:estforcubes} for cubes to conclude that \eqref{eq:stdlpest} holds. 
\end{rmk}

\section{Airy equation}\label{sec:airy}
We finish our exposition with a discussion of some partial progress for the Airy equation
\begin{equation}\label{eq:airy}
	\partial_t u - \partial^3_{xxx} u = 0
\end{equation}
where $u:\mathbb{R}\times\mathbb{R}\to \mathbb{R}$. As we have seen previously, the classical analogue of this equation fails to exhibit any decay. On the other hand, by oscillatory integral techniques it is known that solutions enjoy a decay estimate of the form \cite{PrinAnal4}*{Chapter 8}
\begin{equation}\label{eq:stdairydecay}
	|t|^{\frac13} |u(t,x)| \leq C \|u(0,\text{---})\|_{L^1(\mathbb{R})}.
\end{equation}
The question is: can decay for the Airy equation be proven using commuting-vector-field techniques? Here, we show how to recover some decay estimates using only commuting differential operators. 

Returning to the proof of Theorem \ref{thm:degenerate}, we see that it is possible to construct classical commuting vector fields by using analogues of Galilean symmetry and the dispersion relation function $w$. The quantum analogue (which we have already used in studying the Schr\"odinger equation) has an easy interpretation. Let $P$ denote some real polynomial and consider the equation
\[ i\partial_t u +  P(i\partial_x) u = 0.\]
Taking the formal space-time Fourier transform we expect solutions to be Fourier transforms of measures supported on the surface
\[ \Sigma_P := \{\tau + P(\xi) = 0 \}. \]
Now let $V$ be any vector field in on Fourier space that is tangent to $\Sigma_P$, then $V(\tilde{u})$ is, at least formally, another measure supported on $\Sigma_P$, and hence the operator corresponding to acting by $V$ on the Fourier side is expected to commute with the evolution equation. 

\begin{rmk}
	This same idea has been previously used by Chen and Zhou to derive decay estimates for hyperbolic systems via \emph{pseudodifferential} commutators \cite{FourierVectorField}. More recently  Donninger and Krieger studied equations with potential via a distorted Fourier transform, and proved decay estimates using operators build also from vector fields on the distorted Fourier side.
\end{rmk}

In the case $P(z) = z^2$, we have Schr\"odinger's equation. The differential of the defining function of $\Sigma_P$ is $\mathrm{d}\tau + 2 \xi \mathrm{d}\xi$, and hence the vector field $2\xi \partial_\tau - \partial_\xi$ is tangent to $\Sigma_P$. Taking the Fourier transform we have that this corresponds to the operator $2t \partial_x + i x$ which we used to prove Theorem \ref{thm:displs}. 

The Airy equation \eqref{eq:airy} corresponds to $P(z) = z^3$. The same procedure yields the tangent vector field $3 \xi^2 \partial_\tau - \partial_\xi$ in Fourier space, which corresponds to the differential operator 
\begin{equation}
	W:= 3 t \partial^2_{xx} +  x
\end{equation}
on the physical side, which we can check to indeed commute with \eqref{eq:airy}. With this operator we can prove a space-time weighted $L^\infty$ estimate as follows. Observe that
\begin{align*}
	 \frac32 t [\partial_x u(t,x)]^2 &= \int_{-\infty}^x 3 \partial_x u(t,x') \partial^2_{xx} u(t,x') ~\mathrm{d}x' \\
	 & = \int_{-\infty}^x  \partial_x u(t,x') Wu(t,x') ~\mathrm{d}x' - \int_{-\infty}^x x' u(t,x') \partial_x u(t,x') ~\mathrm{d}x' \\
	 &= \int_{-\infty}^x  \partial_x u(t,x') Wu(t,x') ~\mathrm{d}x' -  \frac12 \int_{-\infty}^x x' \partial_x [u(t,x')]^2 ~\mathrm{d}x' \\
	 &= \int_{-\infty}^x  \partial_x u(t,x') Wu(t,x') ~\mathrm{d}x' -  \frac12 x u(t,x)^2 + \frac12 \int_{-\infty}^x [u(t,x')]^2 ~\mathrm{d}x'.
\end{align*}
Reorganize the terms we obtain
\begin{equation}\label{eq:airypre}
\begin{aligned}
	3t [ \partial_x u(t,x)]^2 + x [u(t,x)]^2 & = 2\int_{-\infty}^x  \partial_x u(t,x') Wu(t,x') ~\mathrm{d}x'  + \int_{-\infty}^x [u(t,x')]^2 ~\mathrm{d}x'\\
& \leq 2\|\partial_x u(t,\text{---})\|_{L^2} \|Wu(t,\text{---})\|_{L^2} + \|u(t,\text{---})\|_{L^2}^2.
\end{aligned}
\end{equation}
The terms to the right of the inequality are all conserved in time, due to the $L^2$ conservation property of the Airy equation. In other words, we have proven 
\begin{prop}\label{prop:airydisp1}
	If $u$ solves \eqref{eq:airy}, then 
	\[ 3t [ \partial_x u(t,x)]^2 + x [u(t,x)]^2 \leq 2\|\partial_x u(0,\text{---})\|_{L^2} \|xu(0,\text{---})\|_{L^2} + \|u(0,\text{---})\|_{L^2}^2.\]
\end{prop}
The above proposition is a far-cry from the estimate \eqref{eq:stdairydecay}. In fact, we are not able to recover exactly the standard decay estimate \eqref{eq:stdairydecay}; below we will show how to get a similar local energy decay statement with the correct scaling. 
But first, let us examine some properties of Proposition \ref{prop:airydisp1}. 
Using the conservation of $L^2$ we can easily obtain \emph{uniform boundedness} 
\[ |u(t,x)|^2 \lesssim \|\partial_x u(0,\text{---})\|_{L^2} \|u(0,\text{---})\|_{L^2}.\]
And hence for any $x_0$, Proposition \ref{prop:airydisp1} implies \emph{uniform} decay in forward time of $|\partial_x u(t,x)|$ for $x \geq x_0$, with rate $t^{1/2}$. This can be explained heuristically by the fact that, since $\partial_{xx}^2$ is a negative operator on $L^2$, we expect (drawing connection to the classical picture) that the corresponding wave-packets for the Airy equation should move to the left, and hence pointwise decay on any right half line should be easier to prove. 
The decay rate of $t^{1/2}$ is correct, in terms of scaling, based on \eqref{eq:stdairydecay}. By using the fundamental solution one can obtain the following decay estimates for the Airy equation:
\[ |\partial_x u(t,x)| \lesssim |t|^{-1/3} \|\partial_x u(0,\text{---})\|_{L^1}, \qquad |\partial_x u(t,x)| \lesssim |t|^{-2/3} \|u(0,\text{---})\|_{L^1}.\]
From these we obtain the interpolated estimate
\[ |\partial_x u(t,x)|^2 \lesssim |t|^{-1} \|\partial_x u(0,\text{---})\|_{L^1}\| u(0,\text{---})\|_{L^1} \]
the right hand side of which have the same scaling as the right hand side which appears in Proposition \ref{prop:airydisp1}. Multiplying this inequality by $\langle x\rangle^{-1-2\epsilon}$ and integrating, we obtain as a consequence the local energy decay 
\[ |t| \, \| \langle \text{---} \rangle^{-\frac12 - \epsilon} \partial_x u(t,\text{---}) \|_{L^2}^2 \lesssim  \|\partial_x u(0,\text{---})\|_{L^1}\| u(0,\text{---})\|_{L^1}.\]
A similar estimate (with the same scaling) can be derived as a consequence of Proposition \ref{prop:airydisp1}.
\begin{cor}
	If $u$ solves \eqref{eq:airy}, then 
	\[ |t|\, \|\langle\text{---}\rangle^{-\frac12-\epsilon} \partial_x u(t,\text{---})\|_{L^2}^2 \lesssim_{\epsilon} \|\partial_x u(0,\text{---})\|_{L^2} \|\langle\text{---}\rangle u(0,\text{---})\|_{L^2} + \|u(0,\text{---})\|_{L^2}^2.\]
\end{cor}

\begin{proof}
	Multiply the inequality in Proposition \ref{prop:airydisp1} by $\langle x\rangle^{-1-2\epsilon}$ and integrate in $x$, noting that the weight is integrable. The proof concludes by noting that 
	\[ \left| \int_{\mathbb{R}} \frac{x}{\langle x\rangle^{-1-2\epsilon}} |u(t,x)|^2 ~\mathrm{d}x\right| \leq \|u(t,\text{---})\|_{L^2}^2 = \|u(0,\text{---})\|_{L^2}^2. \]
\end{proof}

\begin{rmk}
	We are able to obtain a correctly-scaled local-energy decay estimate for $\partial_x u$. It remains open whether a correctly-scaled decay estimate for $u$ itself is possible using a commuting vector field approach. If one takes the point of view as above where the commuting linear operators used correspond to tangential vector fields on the Fourier side, then the answer seems to be in the negative. This is based on the fact that tangential vector fields on the Fourier side whose Fourier transforms are differential operators can have only weights in integer powers of $t$. Coupled with $L^2$ based conservation laws this suggests that only estimates with decay in the order of $t^{-k/2}$ where $k$ is an integer is possible via this technique. Hence our result for the Airy equation which scales the same as the decay estimate that is interpolated between $L^1$---$W^{1,\infty}$ decay and $W^{1,1}$---$W^{1,\infty}$ decay. 

	For another example, one can also consider equations of the form 
	\[ i\partial_t u + \partial^{2k}_{xx\cdots x} u = 0.\]
	Fourier techniques give decay rates of the form 
	\[ |t|^{1/2k} \|u(t,\text{---})\|_{L^\infty} \lesssim \|u(0,\text{---})\|_{L^1}.\]
	Running the same argument essentially as in the case of Section \ref{sec:sch} with the linear operator 
	\[ W = 2k t \partial^{2k-1}_{xx\cdots x} \pm i x \]
	we obtain a (correctly scaled) estimate of the form
	\[ t |\partial^{2k-2}_{xx\cdots x} u(t,x)|^2 \lesssim \|\partial^{2k-2}_{xx\cdots x} u(t,x)\|_{L^2} \| x u\|_{L^2}.\]

	It remains conceivable that estimates of the lower-order derivative terms can be achieved by a commuting-operator approach. For that to hold, however, one would likely need to allow the commuting operator to be pseudo-differential on one or both of the physical and Fourier sides. And this, in a way, defeats the purpose of this exercise. 
\end{rmk}


\begin{bibdiv}
\begin{biblist}

	\bib{Interpolation}{book}{
		title={Interpolation Spaces},
		subtitle={An Introduction},
		author={Bergh, J\"oran},
		author={L\"ofstr\"om, J\"orgen},
		publisher={Springer-Verlag},
		address={New York},
		series={Grundlehren der mathematischen Wissenschaften},
		volume={223},
		date={1976}
	}
	
	\bib{FourierVectorField}{article}{
		title={Decay rate of solutions to hyperbolic system of first order},
		author={Chen, Shuxing},
		author={Zhou, Yi},
		journal={Acta Math. Sin. (Engl. Ser.)},
		date={1999},
		volume={15},
		pages={471--484}
	}

	\bib{Distorted}{article}{
		title={A vector field on the distorted Fourier side and decay for wave equations with potentials},
		author={Donninger, Roland},
		author={Krieger, Joachim},
		journal={Mem. Amer. Math. Soc.},
		volume={241},
		date={2016},
		number={1142}
	}

\bib{VectorfieldVlasov}{article}{
	title={A vector field method for relativistic transport equations with applications},
	author={Fajman, David},
	author={Joudioux, J\'er\'emie},
	author={Smulevici, Jacques},
	date={2015},
	eprint={arXiv:1510.04939}
}

	\bib{KeelTao}{article}{
		title={Endpoint Strichartz Estimate},
		author={Keel, Markus},
		author={Tao, Terence},
		journal={Amer. J. Math.},
		date={1998},
		volume={120},
		pages={955--980}
	}

\bib{KlainermanSobolev}{article}{
	title={Uniform decay estimates and the {L}orentz invariance of the classical wave equation},
	author={Klainerman, Sergiu},
	date={1985},
	journal={Comm. Pure Appl. Math.},
	volume={38},
	pages={321--332}
}

\bib{KlainermanKG}{article}{
	title={Global existence of small amplitude solutions to nonlinear {K}lein-{G}ordon equations in four space-time dimensions},
	author={Klainerman, Sergiu},
	date={1985},
	journal={Comm. Pure Appl. Math.},
	volume={38},
	pages={631--641}
}


\bib{KlainermanCommuting}{article}{
	title={A commuting vectorfields approach to {S}trichartz-type inequalities and applications to quasi-linear wave equations},
	author={Klainerman, Sergiu},
	date={2001},
	journal={Int. Math. Res. Not. IMRN},
	volume={2001},
	pages={221--274}
}

\bib{PhysSpace}{article}{
	author={Klainerman, Sergiu},
   author={Rodnianski, Igor},
   author={Tao, Terence},
   title={A physical space approach to wave equation bilinear estimates},
   note={Dedicated to the memory of Thomas H. Wolff},
   journal={J. Anal. Math.},
   volume={87},
   date={2002},
   pages={299--336}
}

	\bib{Hyperboloidal}{book}{
		title={The Hyperboloidal Foliation Method},
		author={LeFloch, Philippe G.},
		author={Ma, Yue},
		date={2015},
		publisher={World Scientific},
		address={Hackensack, NJ},
		series={Series in Applied and Computational Mathematics},
		volume={2}
	}


	\bib{SchlagMusculu}{book}{
		title={Classical and Multilinear Harmonic Analysis, Volume 1},
		author={Muscalu, Camil},
		author={Schlag, Wilhelm},
		date={2013},
		publisher={Cambridge University Press},
		address={New York}
	}

\bib{PlanchonVega}{article}{
	author={Planchon, Fabrice},
	author={Vega, Luis},
	   title={Bilinear virial identities and applications},
	   language={English, with English and French summaries},
	   journal={Ann. Sci. \'Ec. Norm. Sup\'er. (4)},
	   volume={42},
	   date={2009},
	   number={no.~2},
	   pages={261--290}
}

\bib{JacquesVlasov}{article}{
	author={Smulevici, Jacques},
	title={Small data solutions of the Vlasov-Poisson system and the vector field method},
	journal={Annals of PDE},
	volume={2},
	date={2016},
	eprint={arXiv:1504.02195}
}

	\bib{Stein3}{book}{
		title={Harmonic Analysis},
		author={Stein, Elias M.},
		date={1995},
		publisher={Princeton University Press},
		address={Princeton, NJ},
		subtitle={Real-Variable Methods, Orthogonality, and Oscillatory Integrals}
	}


\bib{PrinAnal4}{book}{
	title={Functional Analysis}, 
	author={Stein, Elias M.},
	author={Shakarchi, Rami},
	date={2011},
	publisher={Princeton University Press},
	address={Princeton, NJ},
	volume={4},
	series={Princeton Lectures in Analysis},
	subtitle={Introduction to Further Topics in Analysis}
}

\bib{TaoUnp}{article}{
	title={A physical space proof of the bilinear Strichartz and local smoothing estimates for the Schrodinger equation},
	author={Tao, Terence},
	date={2013},
	status={unpublished},
	eprint={https://terrytao.wordpress.com/2013/07/10/}
}

\bib{QianPP}{article}{
	title={An intrinsic hyperboloid approach for Einstein Klein-Gordon equations},
	author={Wang, Qian},
	date={2016},
	status={preprint},
	eprint={arXiv:1607.01466}
}

\end{biblist}
\end{bibdiv}
\end{document}